\documentclass[reqno]{amsart}

\usepackage{amsfonts}
\usepackage{amsmath}
\usepackage{amssymb}

\newtheorem{theorem}{Theorem}
\newtheorem{definition}[theorem]{Definition}
\newtheorem{example}[theorem]{Example}
\newtheorem{remark}[theorem]{Remark}

\allowdisplaybreaks

\AtBeginDocument{{\noindent\small
This is a preprint of a paper whose final and definite form is with 
\emph{Commun. Fac. Sci. Univ. Ank. Ser. A1 Math. Stat.}, ISSN: 1303-5991. 
Submitted 14 May 2018; Article revised 18 Nov 2018;
Article accepted for publication 23 Nov 2018.} 
\vspace{9mm}}

\begin{document}
	
\title[Structural derivatives on time scales]{Structural derivatives on time scales}

\author[B. Bayour]{Benaoumeur Bayour}
\address{Benaoumeur Bayour: University of Mascara, B.P. 305, Mamounia Mascara, Mascara 29000, Algeria.}
\email{b.bayour@univ-mascara.dz}

\author[D. F. M. Torres]{Delfim F. M. Torres}
\address{Delfim F. M. Torres: University of Aveiro,
Center for Research and Development in Mathematics and Applications (CIDMA),
Department of Mathematics, 3810-193 Aveiro, Portugal.}
\email{delfim@ua.pt}
\urladdr{http://orcid.org/0000-0001-8641-2505}


\date{Received: May 14, 2018; Revised: Nov 18, 2018; Accepted: Nov 23, 2018}



\thanks{This paper is in final form and no version of it will be submitted
for publication elsewhere.}


\begin{abstract}
We introduce the notion of structural derivative on time scales.
The new operator of differentiation unifies the concepts of fractal 
and fractional order derivative and is motivated by lack of classical 
differentiability of some self-similar functions. Some properties 
of the new operator are proved and illustrated with examples.
\end{abstract}

\maketitle


\section{Introduction}

In the past few years, several operators of differentiation 
have been investigated by researchers, from almost all branches
of sciences, technology, and engineering, due to their capabilities
of better modeling and predict complex systems \cite{ssp,wer,yba}.
In \cite{wc2}, the concept of \emph{Hausdorff derivative} of a function $f(t)$ 
with respect to a fractal measure of $t$ is introduced:
\begin{equation}
\label{fractal:der}
\frac{d f(t)}{d t^{\alpha}}
=\lim_{s\mapsto t} \frac{f(t)-f(s)}{t^{\alpha}-s^{\alpha}}.
\end{equation}
In order to describe a rather large number of experimental results in Biomedicine,
related to the structure of the diffusion of magnetic resonance imaging signals in
human brain regions, the \emph{structural derivative} is defined in \cite{yba} as
\begin{equation}
\label{eq:SD}
\frac{d f(t)}{d_{p} t}
=\lim_{s\mapsto t}\frac{f(t)-f(s)}{p(t)-p(s)},
\end{equation}
where $p(\cdot)$ is the structural function. 
When $p(t)= t^{\alpha}$, the structural derivative \eqref{eq:SD}
coincides with the \emph{fractal derivative} \eqref{fractal:der},
as called in \cite{MR3672032}. It is important to emphasize 
that the structural function $p(t)$ is not necessarily 
a power function. Examples in the literature can be found where $p(t)$
is the inverse Mittag--Leffler function, the probability density function, 
or the stretched exponential function \cite{wc2}. Compared with classical nonlinear models, 
structural differential equations require fewer parameters and lower computational costs 
in detecting causal relationships between mesoscopic time-space structures 
and certain physical behaviors \cite{wh2}. 

Here we generalize the important notion of structural derivative 
to an arbitrary time scale $\mathbb{T}$. As particular cases,
we get the fractional order derivative \cite{KAK}
and the fractional derivative on time scales recently introduced in \cite{MR33}. 
Moreover, we claim that the new \emph{structural derivative on time scales} 
is more than a mathematical generalization, 
allowing to deal with important concepts 
that may appear in complex systems, such as 
self-similarity and non-differentiability 
(see Example~\ref{ex:struct:der:self:sim:f}).

The need for the structural derivative notion, on different time scales 
than the set of real numbers, appears naturally in complex coarse-graininess 
structures, for instance, in anomalous radiation absorption where a tumor 
tissue interacts with the media and the radiation \cite{Weberszpil}.
In such coarse-grained spaces, a point is not infinitely thin, and this 
feature is better modeled by means of our time-scale structural derivative. 
Indeed, in our approach we include a scale in time, allowing to consider 
the effects of internal times on the systems. For examples on the usefulness 
of structural derivatives on the quantum time scale, to model complex systems 
on life, medical, and biological sciences, see also \cite{Weberszpil:Lazo}.
Our structural derivative on time scales allows to unify different  
structural derivatives found in the literature in specific time scales, 
as in the continuous, the discrete, and the quantum scales.

The paper is organized as follows. In Section~\ref{sec:2}, 
we briefly recall the necessary concepts from the time-scale calculus.
Then, in Section~\ref{sec:3}, we introduce the new structural derivative 
on time scales and prove its main proprieties. Illustrative examples 
are given along the text. In Section~\ref{rem:self:sim},
we remark that a function can be structural differentiable 
on a general time scale without being differentiable.
We end with Section~\ref{sec:conc} of conclusions 
and possible future work.


\section{Preliminaries}
\label{sec:2}

A time scale $\mathbb{T}$ is an arbitrary nonempty
closed subset of the real numbers $\mathbb{R}$.
For $t\in\mathbb{T}$, we define the forward jump operator
$\sigma:\mathbb{T}\rightarrow\mathbb{T}$
by $\sigma(t)=\inf\{s\in\mathbb{T}:s>t\}$
and the backward jump operator $\rho:\mathbb{T}\rightarrow\mathbb{T}$
by $\rho(t):=\sup\{s\in\mathbb{T}:s<t\}$.
Then, one defines the graininess function
$\mu:\mathbb{T}\rightarrow[0,+\infty[$
by $\mu(t)= \sigma(t)-t$.
If $\sigma(t)>t$, then we say that $t$ is right-scattered;
if $\rho(t)<t $, then $t$ is left-scattered. Moreover,
if $t < \sup\mathbb{T}$ and $\sigma(t)=t$, then $t$ is called right-dense;
if $t>\inf\mathbb{T}$ and $\rho(t)=t$, then $t$ is called left-dense.
If $\mathbb{T}$ has a left-scattered maximum $m$, then we define
$\mathbb{T}^{\kappa}=\mathbb{T}\setminus\{ m\}$; otherwise
$\mathbb{T}^{\kappa}=\mathbb{T}$.
If $f:\mathbb{T}\rightarrow \mathbb{R}$, then
$f^{\sigma}:\mathbb{T}\rightarrow \mathbb{R}$ is given by
$f^{\sigma}(t)=f(\sigma(t))$ for all $t\in\mathbb{T}$.

\begin{definition}[The Hilger derivative \cite{MBP}]
\label{def:Hder}
Let $f:\mathbb{T}\rightarrow \mathbb{R}$ and $t\in\mathbb{T}$.
We define $f^{\Delta}(t)$ to be the number,
provided it exists, with the property that given any
$\epsilon>0$ there is a neighborhood $U$ of $t$
(i.e., $U=(t-\delta,t+\delta)\cap\mathbb{T}$ for some $\delta>0$) such that
$$
|[f(\sigma(t))-f(s)]-f^{\Delta}(t)[\sigma(t)-s]|\leq\epsilon|\sigma(t)- s|
$$
for all $s\in U$. We call $f^{\Delta}(t)$ the Hilger (or the time-scale)
derivative of $f$ at $t$.
\end{definition}

For more on the calculus on time scales,
we refer the reader to the books \cite{MR1843232,MBP}.


\section{Structural derivatives on time scales}
\label{sec:3}

We introduce the definition
of structural derivative on time scales.
Here we follow the delta/forward approach. 
However, it should be mentioned that such choice is not fundamental 
for our structural derivative notion on time scales. 
In particular, the nabla/backward approach is also possible, with 
the properties we prove here for the delta derivative calculus
being easily mimicked to the nabla case, where instead of using 
the forward $\sigma(t)$ operator we use the backward 
$\rho(t)$ operator of time scales. 

\begin{definition}[The time-scale structural derivative]
\label{def:ourStruct:Der:TS}
Assume $f,p:\mathbb{T}\rightarrow \mathbb{R}$ with
$\mathbb{T}$ a time scale. Let $t\in \mathbb{T}^{\kappa}$
and $\lambda>0$. We define $f^{\Delta_{p}^{\lambda}}(t)$
to be the number, provided it exists, with the property that given
any $\epsilon >0$ there is a neighborhood $U$ of $t$
(i.e., $U=(t-\delta,t+\delta)\cap \mathbb{T}$
for some $\delta>0$) such that
\begin{equation*}
\left|[f^{\lambda}(\sigma(t))-f^{\lambda}(s)]
-f^{\Delta_{p}^{\lambda}}(t)[p(\sigma(t))-p(s)]\right|
\leq \epsilon \left|p(\sigma(t))-p(s)\right|
\end{equation*}
for all $s\in U$. We call $f^{\Delta_{p}^{\lambda}}(t)$
the structural derivative  of $f$ at $t$ (associated with
$\lambda$ and $p(\cdot)$). Moreover, we say that $f$ 
is structural differentiable on $\mathbb{T}^{\kappa}$ 
(or $\Delta_{p}^{\lambda}$-differentiable), provided $f^{\Delta_{p}^{\lambda}}(t)$ 
exists for all $t\in\mathbb{T}^{\kappa}$.
\end{definition}

It is clear that if $\lambda = 1$ and $p(t)=t$, then the new derivative coincides 
with the standard Hilger derivative (i.e., Definition~\ref{def:ourStruct:Der:TS}
reduces to Definition~\ref{def:Hder}). Our first result shows, in particular, 
that for $\mathbb{T} = \mathbb{R}$ and $\lambda = 1$, we obtain from
Definition~\ref{def:ourStruct:Der:TS} the structural derivative \eqref{eq:SD}.

\begin{theorem}
\label{th1}
Assume $f, p:\mathbb{T}\rightarrow \mathbb{R}$ with $\mathbb{T}$
a time scale. Let $t\in \mathbb{T}^{\kappa}$ and $\lambda\in\mathbb{R}$, 
$\lambda>0$. Then the following proprieties hold:
\begin{enumerate}
\item\label{2} If $f$ is continuous at $t$ and $t$ is right-scattered,
then $f$ is  $\Delta_{p}^{\lambda}$-differentiable  at $t$ with
\begin{equation}
\label{eq1}
f^{\Delta_{p}^{\lambda}}(t)
=\frac{f^{\lambda}(\sigma(t))-f^{\lambda}(t)}{p(\sigma(t))-p(t)}.
\end{equation}
\item \label{3} If $t$ is right-dense, then $f$ is structural differentiable
at $t$ if and only if the limit
$$
\lim_{s\rightarrow t}\frac{f^{\lambda}(t)-f^{\lambda}(s)}{p(t)-p(s)}
$$
exists as a finite number. In this case,
\begin{equation}
\label{eq2}
f^{\Delta_{p}^{\lambda}}(t)
=\lim_{s\rightarrow t}\frac{f^{\lambda}(t)-f^{\lambda}(s)}{p(t)-p(s)}.
\end{equation}
\item If $f$ is structural differentiable at $t$, then
$$
f^{\lambda}(\sigma(t))
=f^{\lambda}(t)+(p(\sigma(t))-p(t))f^{\Delta_{p}^{\lambda}}(t).
$$
\end{enumerate}
\end{theorem}

\begin{proof}
(1)  Assume $f$ is continuous at $t$ with $t$ right-scattered.
By continuity,
$$
\lim_{s\rightarrow t}\frac{f^{\lambda}(\sigma(t))
-f^{\lambda}(s)}{p(\sigma(t))-p(s)}
=\frac{f^{\lambda}(\sigma(t))
-f^{\lambda}(t)}{p(\sigma(t))-p(t)}.
$$
Hence, given $\epsilon>0$, there is a neighborhood $U$ of $t$
such that
$$
\left|\frac{f^{\lambda}(\sigma(t))
-f^{\lambda}(s)}{p(\sigma(t))-p(s)}
-\frac{f^{\lambda}(\sigma(t))
-f^{\lambda}(t)}{p(\sigma(t))-p(t)}\right|
\leq\epsilon
$$
for all $s\in U$. It follows that
$$
\left|f^{\lambda}(\sigma(t))
-f^{\lambda}(s)-\frac{f^{\lambda}(\sigma(t))
-f^{\lambda}(t)}{p(\sigma(t))-p(t)}[p(\sigma(t))-p(s)] \right|
\leq \epsilon\left|p(\sigma(t))-p(s) \right|
$$
for all $s\in U$. Hence, we get the desired equality \eqref{eq1}.

(2) Assume $f$ is structural differentiable at $t$ and $t$ is right-dense.
Let $\epsilon>0$ be given. Since $f$ is differentiable at $t$,
there is a neighborhood $U$ of $t$ such that
$$
\mid[f^{\lambda}(\sigma(t))-f^{\lambda}(s)] -f^{\Delta_{p}^{\lambda}}(t)
\left[p(\sigma(t))-p(s) \right]\mid
\leq\epsilon\mid p(\sigma(t))-p(s) \mid
$$
for all $s\in U$. Moreover, because $\sigma(t)=t$, we have that
$$
\mid[f^{\lambda}(t)-f^{\lambda}(s)] 
-f^{\Delta_{p}^{\lambda}}(t)[p(t)
-p(s)]\mid \leq\epsilon\mid p(t)-p(s) \mid
$$
for all $s\in U$. It follows that
$
\left|\frac{f^{\lambda}(t) -f^{\lambda}(s)}{p(t)
-p(s)}-f^{\Delta_{p}^{\lambda}}(t) \right|
\leq\epsilon
$
for all $s\in U$, $s\neq t$, and we get equality \eqref{eq2}. Assume
$
\lim_{s\rightarrow t}
\frac{f^{\lambda}(t)-f^{\lambda}(s)}{p(t)-p(s)}
$
exists and is equal to $\xi$ and $\sigma(t)=t$. Let $\epsilon>0$.
Then there is a neighborhood $U$ of $t$ such that
$$
\left|\frac{f^{\lambda}(\sigma(t))
-f^{\lambda}(s)}{p(t)-p(s)}-\xi\right|
\leq \epsilon
$$
for all $s\in U$. Because
$\mid f^{\lambda}(\sigma(t))
-f^{\lambda}(s)-\xi(p(t)-p(s))\mid
\leq \epsilon |p(t)-p(s)|$
for all $s\in U$,
$$
f^{\Delta_{p}^{\lambda}}(t)=\xi
=\lim_{s\rightarrow t} \frac{f^{\lambda}(t)
-f^{\lambda}(s)}{p(t)-p(s)}.
$$

(3) If $\sigma(t)=t$, then $p(\sigma(t))-p(t)=0$ and
$$
f^{\lambda}(\sigma(t))=f^{\lambda}(t)
=f^{\lambda}(t)+(p(\sigma(t))
-p(t))f^{\Delta_{p}^{\lambda}}(t).
$$
On the other hand, if $\sigma(t)>t$, then by item \ref{2}
\begin{eqnarray*}
f^{\lambda}(\sigma(t))&=&f^{\lambda}(t)+(p(\sigma(t))-p(t))
\frac{f^{\lambda}(\sigma(t))-f^{\lambda}(t)}{p(\sigma(t))-p(t)}\\
&=& f^{\lambda}(t)+(p(\sigma(t))-p(t))f^{\Delta_{p}^{\lambda}}(t)
\end{eqnarray*}
and the proof is complete.
\end{proof}

\begin{example}
If $\mathbb{T}=\mathbb{R}$ and $p(t)=t^{\alpha}$, 
then it follows from Theorem~\ref{th1}
that our structural derivative on time scales reduces
to the generalized fractal-fractional derivative of $f$ 
of order $\alpha$ in \cite{MR3672032}:
$$
f^{\Delta_{p}^{\lambda}}(t)=\lim_{s\rightarrow t}
\frac{f^{\lambda}(t)-f^{\lambda}(s)}{t^{\alpha}-s^{\alpha}}.
$$
\end{example}

\begin{example}
If $\mathbb{T}=\mathbb{R}$, $p(t)=t^{\alpha}$, and $\lambda=\alpha$, 
then item \ref{3} of Theorem~\ref{th1} yields that $f:\mathbb{R}\rightarrow \mathbb{R}$ 
is structural differentiable at $t\in \mathbb{R}$ if, and only if,
$$
f^{\Delta_{p}^{\lambda}}(t)=\lim_{s\rightarrow t}
\frac{f^{\alpha}(t)-f^{\alpha}(s)}{t^{\alpha}-s^{\alpha}}
$$ 
exists. In this case, we get the fractional order derivative
$f^{(\alpha)}$ of \cite{KAK}.
\end{example}

\begin{example}
If $\mathbb{T}=\mathbb{Z}$, then item \ref{2} of Theorem~\ref{th1} yields
that $f:\mathbb{Z}\rightarrow \mathbb{R}$ is structural differentiable
 at $t\in \mathbb{Z}$  with
$$
f^{\Delta_{p}^{\lambda}}(t)=\frac{f^{\lambda}(\sigma(t))
-f^{\lambda}(t)}{p(\sigma(t))-p(t)}
=\frac{f^{\lambda}(t+1)-f^{\lambda}(t)}{p(t+1)-p(t)}.
$$
\end{example}

\begin{example}
\label{example3}
If $f:\mathbb{T}\rightarrow \mathbb{R}$ is defined
by $f(t)\equiv\gamma \in\mathbb{R}$,
then $f^{\Delta_{p}^{\lambda}}(t) \equiv 0$. Indeed,
if $t$ is right-scattered, then by item \ref{2} of Theorem~\ref{th1} we get
$$
f^{\Delta_{p}^{\lambda}}(t)
=\frac{f^{\lambda}(\sigma(t))-f^{\lambda}(t)}{p(\sigma(t))-p(t)}
=\frac{\gamma^{\lambda}-\gamma^{\lambda}}{p(\sigma(t))- p(t)}=0;
$$
if $t$ is right-dense, then by \eqref{eq2} we get
$$
f^{\Delta_{p}^{\lambda}}(t)=\lim_{s\rightarrow t}
\frac{\gamma^{\lambda}-\gamma^{\lambda}}{p(t)- p(s)}=0.
$$
\end{example}

\begin{example}
\label{ex1:1}
If $f:\mathbb{T}\rightarrow \mathbb{R}$ is given by $f(t) = t$,
then $f^{\Delta_{p}^{\lambda}}(t) \neq 1$ because
if $\sigma(t)>t$ (i.e., $t$ is right-scattered), then
$$
f^{\Delta_{p}^{\lambda}}(t)=\frac{f^{\lambda}(\sigma(t))
-f^{\lambda}(t)}{p(\sigma(t))-p(t)}
=\frac{\sigma^{\lambda}(t)-t^{\lambda}}{p(\sigma(t))- p(t)}\neq1;
$$
if $\sigma(t)=t$ (i.e., $t$ is right-dense), then
$$
f^{\Delta_{p}^{\lambda}}(t)=\lim_{s\rightarrow t}\frac{f^{\lambda}(t)
-f^{\lambda}(s)}{p(t)- p(s)}
=\lim_{s\rightarrow t}\frac{t^{\lambda}-s^{\lambda}}{p(t)- p(s)}\neq 1.
$$
\end{example}

\begin{example}
\label{ex1:2}
Let $g:\mathbb{T}\rightarrow \mathbb{R}$ be given by
$g(t) = \frac{1}{t}$. We have 
$$
g^{\Delta_{p}^{\lambda}}(t)
= - \frac{(t)^{\Delta_{p}^{\lambda}}}{\left(\sigma(t)t\right)^{\lambda}}.
$$
Indeed, if $\sigma(t)=t$, then $g^{\Delta_{p}^{\lambda}}(t)
=-\frac{(t)^{\Delta_{p}^{\lambda}}}{\sigma^{\lambda}(t)t^{\lambda}}$;
if $\sigma(t)>t$, then
$$
g^{\Delta_{p}^{\lambda}}(t)
=\frac{g^{\lambda}(\sigma(t))-g^{\lambda}(t)}{p(\sigma(t))- p(t)}
= \frac{\left(\frac{1}{\sigma(t)}\right)^{\lambda}
-\left(\frac{1}{t}\right)^{\lambda}}{p(\sigma(t))- p(t)}
=\frac{\frac{t^{\lambda}-\sigma^{\lambda}(t)}{t^{\lambda}
\sigma^{\lambda}(t)}}{t^{\lambda}
-\sigma^{\lambda}(t)} = - \frac{(t)^{\Delta_{p}^{\lambda}}}{t^{\lambda}\sigma^{\lambda}(t)}.
$$
\end{example}

\begin{example}
\label{ex1:3}
Let $h:\mathbb{T}\rightarrow \mathbb{R}$ be defined by $h(t) = t^{2}$.
We have 
$$
h^{\Delta_{p}^{\lambda}}(t)=(t)^{\Delta_{p}^{\lambda}}(\sigma(t)+t).
$$
Indeed, if $t$ is right-dense, then 
$h^{\Delta_{p}^{\lambda}}(t)=\lim_{s\rightarrow t}
\frac{t^{2\lambda}-s^{2\lambda}}{p(t)- p(s)} 
= (t)^{\Delta_{p}^{\lambda}}(\sigma(t)+t)$; 
if $t$ is right-scattered, then
$$
h^{\Delta_{p}^{\lambda}}(t)=\frac{h^{\lambda}(\sigma(t))
-h^{\lambda}(t)}{p(\sigma(t))- p(t)}=\frac{\sigma^{2\lambda}(t)
-t^{2\lambda}}{p(\sigma(t))- p(t)}=(t)^{\Delta_{p}^{\lambda}}(\sigma(t)+t).
$$
\end{example}

\begin{example}
\label{ex7}
Consider the time scale $\mathbb{T}=h\mathbb{Z}$, $h > 0$.
Let $f$ be the function defined by
$f:h\mathbb{Z}\rightarrow \mathbb{R}$, $t\mapsto (t-c)^2$, $c\in\mathbb{R}$.
The time-scale structural derivative of $f$ at $t$ is
\begin{equation*}
\begin{split}
f^{\Delta_{p}^{\lambda}}(t)
&=\frac{f^{\lambda}(\sigma(t))-f^{\lambda}(t)}{p(\sigma(t))-p(t)}
=\frac{((\sigma(t)-c)^{2})^{\lambda}
-((t-c)^{2})^{\lambda}}{p(\sigma(t))- p(t)}\\
&= \frac{(t+h-c)^{2\lambda}-(t-c)^{2\lambda}}{p(t+h)-p(t)}.
\end{split}
\end{equation*}
\end{example}

\begin{remark}
Our examples show that, in general, 
$f^{\Delta_{p}^{\lambda}}(t)$ is a complex number
(for instance, choose  $\lambda=\frac{1}{2}$ and $t<0$).
\end{remark}

Our second theorem shows that it is possible to develop 
a calculus for the time-scale structural derivative.

\begin{theorem}
\label{th2}
Assume $f,g:\mathbb{T}\rightarrow\mathbb{R}$ are continuous
and structural differentiable  at $t\in \mathbb{T}^{\kappa}$.
Then the following proprieties hold:
\begin{enumerate}
\item For any constant $\gamma$, function
$\gamma f:\mathbb{T}\rightarrow \mathbb{R}$ is structural differentiable
at $t$ with $(\gamma f)^{\Delta_{p}^{\lambda}}(t)
=\gamma^{\lambda} f^{\Delta_{p}^{\lambda}}(t)$.

\item The product $fg:\mathbb{T}\rightarrow \mathbb{R}$ 
is structural differentiable at $t$ with
\begin{equation*}
\begin{split}
(fg)^{\Delta_{p}^{\lambda}}(t)
&=f^{\Delta_{p}^{\lambda}}(t) g^{\lambda}(t)
+f^{\lambda}(\sigma(t)) g^{\Delta_{p}^{\lambda}}(t)\\
&= f^{\Delta_{p}^{\lambda}}(t) g^{\lambda}(\sigma(t))
+f^{\lambda}(t) g^{\Delta_{p}^{\lambda}}(t).
\end{split}
\end{equation*}

\item If $f(t)f(\sigma(t))\neq 0$, then $\frac{1}{f}$
is structural differentiable  at $t$ with
$$
\left(\frac{1}{f}\right)^{\Delta_{p}^{\lambda}}(t)
=\frac{-f^{\Delta_{p}^{\lambda}}(t)}{f^{\lambda}(\sigma(t))f^{\lambda}(t)}.
$$

\item If $g(t)g(\sigma(t))\neq 0$, then
$\frac{f}{g}$  is structural differentiable  at $t$ with
$$
\left(\frac{f}{g}\right)^{\Delta_{p}^{\lambda}}(t)
= \frac{f^{\Delta_{p}^{\lambda}}(t) g^{\lambda}(t)
-f^{\lambda}(t) g^{\Delta_{p}^{\lambda}}(t)}{g^{\lambda}(\sigma(t))g^{\lambda}(t)}.
$$
\end{enumerate}
\end{theorem}

\begin{proof}
(1) Let $\epsilon\in(0,1)$. Define
$\epsilon^{\ast}=\frac{\epsilon}{|\gamma|^{\lambda}}\in(0,1)$. 
Then there exists a neighborhood $U$ of $t$ such that
$$
|f^{\lambda}(\sigma(t))-f^{\alpha}(s)-f^{\Delta_{p}^{\lambda}}(t)(p(\sigma(t))
-p(s))|\leq\epsilon^{\ast} |p(\sigma(t))- p(s)|
$$
for all $s\in U $. It follows that
\begin{equation*}
\begin{split}
|(\gamma f)^{\lambda}&(\sigma(t))
-(\gamma f)^{\lambda}(s)-\gamma^{\lambda} f^{\Delta_{p}^{\lambda}}(t)
(p(\sigma(t)-p(s))|\\
&= |\gamma|^{\lambda}\mid f^{\lambda}(\sigma(t))-f^{\lambda}(s)
-f^{\Delta_{p}^{\lambda}}(t)(p(\sigma(t))-p(s))\mid  \\
&\leq \epsilon^{\ast}|\gamma|^{\lambda}|p(\sigma(t))- p(s)|\\
&\leq \frac{\epsilon}{|\gamma|^{\lambda}}|\gamma|^{\lambda}|
p(\sigma(t))- p(s)|\\
&= \epsilon|p(\sigma(t))- p(s)|
\end{split}
\end{equation*}
for all $s\in U$. Thus, $(\gamma f)^{\Delta_{p}^{\lambda}}(t)
=\gamma^{\lambda} f^{\Delta_{p}^{\lambda}}(t)$ holds.

(2) Let $\epsilon\in (0,1)$. Define 
$$
\epsilon^{\ast}
=\epsilon \left[1+ \left|f^{\lambda}(t)|+|g^{\lambda}(\sigma(t))\right|
+\left|g^{\Delta_{p}^{\lambda}}(\sigma(t))\right|\right]^{-1}. 
$$
Then $\epsilon^{\ast}\in(0,1)$ and there exist neighborhoods
$U_{1},U_{2}$ and $U_{3}$ of $t$ such that
$$
|f^{\lambda}(\sigma(t))-f^{\lambda}(s)-f^{\Delta_{p}^{\lambda}}(t)(p(\sigma(t))
-p(s))|\leq\epsilon^{\ast} |p(\sigma(t))-p(s)|
$$
for all $s\in U_{1}$,
$$
\left|g^{\lambda}(\sigma(t))-g^{\lambda}(s)
-g^{\Delta_{p}^{\lambda}}(t)\left(p(\sigma(t))-p(s)\right)\right| 
\leq \epsilon^{\ast} \left|p(\sigma(t))-p(s)\right|
$$
for all $s\in U_{2}$, and ($f$ is continuous)
$|f(t)-f(s)|\leq\epsilon^{\ast}$ for all $s\in U_{3}$.
Define $U=U_{1}\cap U_{2}\cap U_{3}$ and let $s\in U$. It follows that
\begin{align*}
\Big|&(fg)^{\lambda}(\sigma(t))-(fg)^{\lambda}(s)
-\left[g^{\Delta_{p}^{\lambda}}(t)f^{\lambda}(t)
+ g^{\lambda}(\sigma(t)) f^{\Delta_{p}^{\lambda}}(t)\right][p(\sigma(t))- p(s)]\Big|\\
&= \Big|[f^{\lambda}(\sigma(t))-f^{\lambda}(s)
-f^{\Delta_{p}^{\lambda}}(t)(p(\sigma(t))-p(s))]\left(g^{\Delta_{p}^{\lambda}}(t)\right)
+g^{\lambda}(\sigma(t))f^{\lambda}(s) \\
&\quad + g^{\lambda}(\sigma(t))
f^{\Delta_{p}^{\lambda}}(t)(p(\sigma(t)
-p(s))-f^{\lambda}(s)g^{\lambda}(s)\\
&\quad -\left[g^{\Delta_{p}^{\lambda}}(t)f^{\lambda}(t)
+g^{\lambda}(\sigma(t))f^{\Delta_{p}^{\lambda}}(t)\right] [p(\sigma(t))-p(s) ]\Big|\\
&= \Big|[f^{\lambda}(\sigma(t))-f^{\lambda}(s)
-f^{\Delta_{p}^{\lambda}}(t)(p(\sigma(t))- p(s))](g^{\lambda}(\sigma(t)))\\
&\quad +\left[g^{\lambda}(\sigma(t))-g^{\lambda}(s)-g^{\Delta_{p}^{\lambda}}(t)(p(\sigma(t))
-p(s))\right] f^{\lambda}(t)\\
&\quad +\left[g^{\lambda}(\sigma(t))-g^{\lambda}(s)
-g^{\Delta_{p}^{\lambda}}(t)(p(\sigma(t))-p(s))\right]\left(f^{\lambda}(s)
-f^{\lambda}(t)\right)+ f^{\lambda}(s) g^{\lambda}(s)\\
&\quad +g^{\Delta_{p}^{\lambda}}(t) f^{\lambda}(s)(p(\sigma(t))-p(s))
+ g^{\lambda}(\sigma(t))f^{\Delta_{p}^{\lambda}}(t)(p(\sigma(t))-p(s))
-g^{\lambda}(s)f^{\lambda}(s)\\
&\quad + g^{\Delta_{p}^{\lambda}}(t)
f^{\lambda}(s)(p(\sigma(t))- p(s))
+ g^{\lambda}(\sigma(t)) f^{\Delta_{p}^{\lambda}}(t)(p(\sigma(t))-p(s))
-f^{\lambda}(s)g^{\lambda}(s)\\
&\quad - \left[g^{\Delta_{p}^{\lambda}}(t)f^{\lambda}(t)
+g^{\lambda}(\sigma(t)) f^{\Delta_{p}^{\lambda}}(t)\right][p(\sigma(t))-p(s)]\Big|\\
&\leq \left|f^{\lambda}(\sigma(t))-f^{\lambda}(s)-f^{\Delta_{p}^{\lambda}}(t)
(p(\sigma(t))-p(s))\right| \left|g^{\lambda}(\sigma(t))\right|\\
&\quad + \left|g^{\lambda}(\sigma(t))-g^{\lambda}(s)
-g^{\Delta_{p}^{\lambda}}(t)(p(\sigma(t))
-p(s))\right| \left|f^{\lambda}(t)\right|\\
&\quad +\left|g^{\lambda}(\sigma(t))-g^{\lambda}(s)
-g^{\Delta_{p}^{\lambda}}(t)(p(\sigma(t))-p(s))\right|
\left|f^{\lambda}(s)-f^{\lambda}(t)\right|\\
&\quad +\left|g^{\Delta_{p}^{\lambda}}(t)\right|
\left|f^{\lambda}(t)-f^{\lambda}(s)\right|
\left|p(\sigma(t))-p(s)\right|\\
&= \epsilon^{\ast}\left|g^{\lambda}(\sigma(t))\right|
\left|p(\sigma(t))-p(s)\right|\\
&\quad +\epsilon^{\ast}\left|f^{\lambda}(t)\right|
\left|p(\sigma(t))-p(s)\right|
+\epsilon^{\ast}|p(\sigma(t))-p(s) |\epsilon^{\ast}
+\epsilon^{\ast}\left|g^{\Delta_{p}^{\lambda}}(t)\right|
|p(\sigma(t))-p(s)|\\
&\leq \epsilon^{\ast}
|p(\sigma(t))-p(s)|
\left(\epsilon^{\ast}
+\left|f^{\lambda}(t)\right|
+\left|g^{\Delta_{p}^{\lambda}}(t)\right|
+\left|g^{\Delta_{p}^{\lambda}}(t)\right|\right)\\
&\leq \epsilon^{\ast} \left|p(\sigma(t))-p(s)\right|
\left(1+\left|f^{\lambda}(t)\right|+\left|g^{\Delta_{p}^{\lambda}}(t)\right|
+\left|g^{\Delta_{p}^{\lambda}}(t)\right|\right)
= \epsilon \left|p(\sigma(t))-p(s)\right|.
\end{align*}
Thus $(fg)^{\Delta_{p}^{\lambda}}(t)=f^{\lambda}(t) g^{\Delta_{p}^{\lambda}}(t)
+f^{\Delta_{p}^{\lambda}}(t)g^{\lambda}(\sigma(t))$ holds at $t$.
The other product rule follows from this last
equality by interchanging functions $f$ and $g$.

(3) We use the structural derivative of a constant (Example~\ref{example3}).
Since 
$$
\left(f \cdot \frac{1}{f}\right)^{\Delta_{p}^{\lambda}}(t)=0,
$$
it follows from item 2 that
$$
\left(\frac{1}{f}\right)^{\Delta_{p}^{\lambda}}(t)
f^{\lambda}(\sigma(t)) + f^{\Delta_{p}^{\lambda}}(t)\frac{1}{f^{\lambda}(t)}=0.
$$
Because we are assuming $f(t)f(\sigma(t))\neq 0$, one has
$$
\left(\frac{1}{f}\right)^{\Delta_{p}^{\lambda}}(t)
=\frac{-f^{\Delta_{p}^{\lambda}}(t)}{f^{\lambda}(\sigma(t))f^{\lambda}(t)}.
$$

(4) For the quotient formula we use items 2 and 3 to compute
\begin{equation*}
\begin{split}
\left(\frac{f}{g}\right)^{\Delta_{p}^{\lambda}}(t)
&=\left(f \cdot \frac{1}{g}\right)^{\Delta_{p}^{\lambda}}(t)\\
&=f^{\lambda}(t) \left(\frac{1}{g}\right)^{\Delta_{p}^{\lambda}}(t)
+f^{\Delta_{p}^{\lambda}}(t)\frac{1}{g^{\lambda}(\sigma(t))}\\
&=-f^{\lambda}(t)\frac{g^{\Delta_{p}^{\lambda}}(t)}{g^{\lambda}(\sigma(t))
g^{\lambda}(t)}+f^{\Delta_{p}^{\lambda}}(t)
\frac{1}{g^{\lambda}(\sigma(t))}\\
&=\frac{f^{\Delta_{p}^{\lambda}}(t)g^{\lambda}(t)
-f^{\lambda}(t) g^{\Delta_{p}^{\lambda}}(t)}{g^{\lambda}(\sigma(t))g^{\lambda}(t)}.
\end{split}
\end{equation*}
This concludes the proof.
\end{proof}

\begin{remark}
The structural derivative
of the sum $f+g:\mathbb{T}\rightarrow \mathbb{R}$
does not satisfy the usual property, that is, in general
$(f+g)^{\Delta_{p}^{\lambda}}(t)
\neq f^{\Delta_{p}^{\lambda}}(t) + g^{\Delta_{p}^{\lambda}}(t)$.
For instance, let $\mathbb{T}$ be an arbitrary time scale and
$f,g : \mathbb{T}\rightarrow \mathbb{R}$ be functions defined by 
$f(t)=t$ and $g(t) = 2t$. One can easily find that
$$
(f+g)^{\Delta_{p}^{\lambda}}(t)
=3^{\lambda}\left(\frac{\sigma^{\lambda}(t)-t^{\lambda}}{p(\sigma(t))-p(t)}\right)
\neq f^{\Delta_{p}^{\lambda}}(t)+g^{\Delta_{p}^{\lambda}}(t)
=(1+2^{\lambda})\frac{\sigma^{\lambda}(t)-t^{\lambda}}{p(\sigma(t))-p(t)}.
$$
\end{remark}


\section{A remark on self-similarity and nondifferentiabilty}
\label{rem:self:sim}

In this section, we provide an example where it is natural to define structural 
derivatives on time scales. Precisely, we consider a function that is structural 
differentiable on a general time scale without being differentiable in the classical 
sense. This possibility, to differentiate nonsmooth functions, is very important 
in real world applications, e.g., to deal with models of hydrodynamics continuum 
flows in fractal coarse-grained (fractal porous) spaces, which are discontinuous 
in the embedding Euclidean space \cite{MR2150460}.

A self-similar function is a function that exhibits 
similar patterns when one changes the scale 
of observation: the patterns generated by $f(t)$ 
and $f(a t)$, $a>0$, looks the same. Formally, $f$ is a 
self-similar function of order $\beta$ if it satisfies 
\begin{equation*}
f(a t)=a^{\beta} f(t), \quad a>0, \quad \beta>0,
\end{equation*}
which is interpreted as saying that in the vicinity of $t$ and $at$ 
the function looks the same. A self-similar function $f$ 
obviously satisfies  $f(0)=0$ and, furthermore,  
$$
f(t)=c t^{\beta}, \quad c = f(1).
$$
Let $0<\beta<1$. Then $f(t)$ is clearly not differentiable at $t=0$.
Example~\ref{ex:struct:der:self:sim:f} shows, however, that 
$f(t)$ can be structurally differentiable at $t=0$,
in the sense of our Definition~\ref{def:ourStruct:Der:TS},
on a general time scale.  

\begin{example}
\label{ex:struct:der:self:sim:f}
Let $0<\beta<1$, $\lambda>0$, and $\alpha < \beta \lambda$.
Let $\mathbb{T}$ be any time scale containing the origin, i.e., $0 \in \mathbb{T}$;
$f : \mathbb{T}\rightarrow \mathbb{R}$ be the self-similar function $f(t) = c t^\beta$;  
and choose the structural function $p : \mathbb{T}\rightarrow \mathbb{R}$
to be $p(t)=t^{\alpha}$. It follows from Theorem~\ref{th1} that 
if point $t = 0$ is right-dense, then
\begin{equation*}
\begin{split}
f^{\Delta_{p}^{\lambda}}(0) &=
\lim_{t\mapsto 0}\frac{f^{\lambda}(t)-f^{\lambda}(0)}{p(t)- p(0)} 
= \lim_{t\mapsto 0} \left[c^\lambda t^{\beta\lambda - \alpha}\right]\\
&= 0;
\end{split}
\end{equation*}
if point $t = 0$ is right-scattered, then
\begin{equation*}
\begin{split}
f^{\Delta_{p}^{\lambda}}(0)
&= \frac{f^{\lambda}(\sigma(0))-f^{\lambda}(0)}{p(\sigma(0))- p(0)}\\
&= c^\lambda \sigma(0)^{\beta \lambda - \alpha}.
\end{split}
\end{equation*}
We conclude that $f$ is always $\Delta_{p}^{\lambda}$-differentiable at $t = 0$.
\end{example}


\section{Conclusion}
\label{sec:conc}

We introduced, for the first time, 
the notion of structural derivative on time scales.
The developed calculus allows to unify and extend
several decay models found in the literature.
A nice mathematical example, showing the necessity to define 
structural derivatives on time scales, was given with respect to
self-similar functions in Section~\ref{rem:self:sim}. 
We claim that the new results here obtained
may serve as key tools to model complex systems,
for example in physics, life, and biological
sciences. As future work, one shall develop such
real world models, clearly showing the usefulness
of the structural derivative on time scales. From the
theoretical side, we can proceed with structural 
measures and structural Lebesgue integration.  


\section*{Acknowledgements}

This research was initiated during a one month visit of Bayour 
to the Department of Mathematics of University of Aveiro, Portugal, 
February and March 2018. The hospitality of the host institution 
and the financial support of University of Mascara, Algeria, 
are here gratefully acknowledged. Torres was supported by Portuguese 
funds through CIDMA and FCT, within project UID/MAT/04106/2019. 

The authors are grateful to two anonymous referees, 
for several constructive remarks, suggestions, and questions. 



\end{document}